\documentclass[12pt,oneside,a4paper]{scrartcl}
\usepackage[T1]{fontenc}
\usepackage[utf8]{inputenc}
\usepackage{amssymb, amsmath, latexsym}
\usepackage{amsthm}
\usepackage[austrian,english]{babel}
\usepackage{graphicx}
\usepackage{bm} 
\newcommand{\usemodifiedcc}{\usepackage[boldsans]{ccfonts}}
\usepackage[style=numeric,backend=biber]{biblatex}
\def\dbr{\mathbb{R}}
\def\dbn{\mathbb{N}}
\def\dbq{\mathbb{Q}}

\makeatletter
\def\moverlay{\mathpalette\mov@rlay}
\def\mov@rlay#1#2{\leavevmode\vtop{%
		\baselineskip\z@skip \lineskiplimit-\maxdimen
		\ialign{\hfil$\m@th#1##$\hfil\cr#2\crcr}}}
\newcommand{\charfusion}[3][\mathord]{
	#1{\ifx#1\mathop\vphantom{#2}\fi
		\mathpalette\mov@rlay{#2\cr#3}
	}
	\ifx#1\mathop\expandafter\displaylimits\fi}
\makeatother

\newcommand{\cupdot}{\charfusion[\mathbin]{\cup}{\cdot}}

\theoremstyle{plain}
\newtheorem{thm}{Theorem}
\newtheorem{lem}{Lemma}
\newtheorem{cor}{Corollary}
\theoremstyle{definition}

\newtheorem{rem}{Remark}
\usemodifiedcc
\usepackage{MnSymbol}
\title{%
	Sincov's and other functional equations and negative interest rates}
\author{Gergely Kiss\thanks{The author was supported by a Premium Postdoctoral Fellowship of the Hungarian Academy of Science, and by Hungarian National Research, Development
and Innovation Office grants K-124749 and K-K142993.}~ and Jens Schwaiger
}  
\date{}
\begin{document}  
	\selectlanguage{english}
	\maketitle%
	\begin{quote}\small
		\begin{center}
			\sffamily\bfseries Abstract
		\end{center}
		
		Investigating the future value $F(K,s,t)$ of a capital
  $K$ invested at date $S$ at date $t$ the ``natural'' condition $F(K,s,t)\geq K$ has  lost its naturality because of the strange fact of negative interest rates. This leads to the task of describing the possible solutions of the multiplicative \textsc{Sincov} equation $f(s,u)=f(s,t)f(t,u)$ for $s\leq t\leq u$ where $f(s,t)=0$ may happen. In this paper we solve this task and discuss connections to the theory of investments.
		
		
		
	\end{quote}
\section{Introduction and motivation}\label{sec1}
A rather well known and elegant application of the theory of functional equation  is given by the deduction of the formula  of theoretical interest compounding.
As a starting point some  ``reasonable'' conditions  for the \emph{future value} function
\begin{center}
$F\colon[0,\infty)\times[0,\infty)\to\dbr$
\end{center}
are given:

\begin{align}
	&F(K+L,t)=F(K,t)+F(L,t),\quad K,L,t\geq0\label{condth1}\\
	&F(F(K,t),s)=F(K,t+s),\quad K,s,t\geq 0 \text{ and }\label{condth2}\\
	&F(K,t)\geq K,\quad K,t\geq 0.\label{condth3}
\end{align}

\begin{thm}\label{thm1}
Let $F\colon[0,\infty)\times[0,\infty)\to\dbr$ be given. Then (\ref{condth1}), (\ref{condth2}) and (\ref{condth3}) are satisfied iff there is some  $q\ge1$ such that
\begin{equation}	
F(K,t)=K q^t,\quad K,t\geq 0.	
\end{equation}
\end{thm}
 The proof can be found in \cite[pp. 105--106]{AJ1966}, \cite{EW1978} and in \cite{SJ1988}.

Note that
(\ref{condth3}) together with (\ref{condth1}) implies that $F(K,t)=K\cdot f(t)$, since (\ref{condth1}) says that $F(\cdot,t)$ is additive and (\ref{condth3}) that this function is bounded from below on $[0,\infty)$ (see \cite[p. 34, Theorem 1]{AJ1966}). (\ref{condth2}) implies $f(t+s)=f(t)f(s)$ for all $s,t\ge 0$ and (\ref{condth3}) that $f(t)\geq 1$ for all $t$. This means that $g:=\ln\circ f$ is additive and $\geq 0$ on $[0,\infty)$ and therefore there is some $r\geq 0$ such that $\ln(f(t))=r t$ for all $t$. So $q:=\exp(g(1))=\exp(r)\geq 1$ and $f(t)=q^t$ for all $t$.

\section{Theoretical rule of interest compounding with negative interest rates allowed}\label{sec2}
At least for the last decade it has became common in economics to admit interest rates being zero or even negative. This clearly contradicts \eqref{condth3}. So one could ask for a substitute for Theorem~\ref{thm1} which allows for the new situation.

Note that
\begin{align}
	&F(K,t)\geq c K,\quad K,t\geq 0 \text{ for some } c>0\label{condthm3'}
\end{align}
instead of (\ref{condth3}) does not help. Of course a result as in the theorem would result with some $q>0$. But taking $t$ large enough shows that \eqref{condthm3'} is only possible when $q\geq1$.
One possibility to characterize the theoretical rule of interest compounding if negative interest rates are admissible  could be the following result.
\begin{thm}\label{thm2}
	A function $F\colon[0,\infty)\times[0,\infty)\to\dbr$ satisfies
	\begin{align}
	&F(K+L,t)=F(K,t)+F(L,t),\quad K,L,t\geq0\label{neg1}\\
	&F(F(K,t),s)=F(K,t+s),\quad K,s,t\geq 0\label{neg2}\\
	&F(\cdot,t)\text{ is monotonic for all } t\text{ and}\label{neg3}\\
	&F(K,\cdot) \text{ is monotonic for all } K.\label{neg4}	
	\end{align}
	iff there is some $q\geq0$ such that
		\begin{equation}\label{formula}
		F(K,t)=K q^t,\quad K,t\geq0,
		\end{equation}
	where for $q=0$ both the cases $q^0=1$ and $q^0=0$ are possible.
\end{thm}
\begin{proof}
	Obviously $F$ with (\ref{formula}) satisfies all the conditions (\ref{neg1}) -- (\ref{neg4}).
	
	Let on the other hand $F$ satisfy these conditions. Since $F(\cdot,t)$ is additive and monotonic it is bounded from one side on some interval which implies (see \cite[p. 34, Theorem~1]{AJ1966}) that $F(K,t)=K\cdot f(t)$ with $f(t)=F(1,t)$. Condition (\ref{neg2}) implies with $K=1$ that
	\begin{equation*}
		f(t)f(s)=f(s+t),\quad s,t\geq0.
	\end{equation*}
To solve this we follow \cite[p. 38, Theorem~1]{AJ1966} and
assume that $f(t_0)=0$ for some $t_0>0$. Then $f(t)=f(t_0+(t-t_0))=f(t_0)f(t-t_0)=0$ for all $t\geq t_0$. Since moreover $0=f(t_0)=f\left(n\frac{t_0}{n}\right)=f\left(\frac{t_0}{n}\right)^n$, $t_0$ may be chosen arbitrarily close to $0$. So $f(t)=0$ for all $t>0$. $f(0)=f(0+0)=f(0)^2$ implies $f(0)\in\{0,1\}$ and therefore $f(t)=0^t$ for all $t$ with $0^0\in\{0,1\}$.
In the remaining case $f(t)$ must be different from $0$ for all $t>0$. By $f(t)=f\left(\frac{t}{2}\right)^2$ the value $f(t)$ must even be $>0$. Note also that $f(0)=1$ since $f(t)=f(t+0)=f(t)f(0)$. Moreover $f$ is monotonic. So using the remarks following Theorem~\ref{thm1} there is some $q>0$ such that $f(t)=q^t$ for all $t$.
\end{proof}	
\section{Future value formulas depending on the interval of investment}\label{sec3}
In \cite[Theorem~2]{SJ1988} a situation is discussed where for $\Delta:=\Delta_\dbr:=\{(s,t)\in\dbr^2\mid s\leq t\}$ the value of the function $F\colon[0,\infty)\times\Delta\to[0,\infty)$  at $(K,s,t)$ denotes the future value of the capital
$K$ at time $t$ when invested at time $s$.
Theorem~2 in \cite{SJ1988} reads as follows.
\begin{thm}\label{thm3}
The function $F\colon[0,\infty)\times\Delta\to[0,\infty)$
satisfies the conditions
\begin{align}
&F(K+L,s,t)=F(K,s,t)+F(L,s,t),\quad K,L\geq0,(s,t)\in\Delta\label{delta1}\\
&F(F(K,s,t),t,u)=F(K,s,u), \quad K\geq0, (s,t), (t,u)\in\Delta
\text{ and  }\label{delta2}\\
&F(K,s,t)\geq K, \quad K\geq0, (s,t)\in \Delta,\label{delta3}
\end{align}
iff there is some non decreasing function $\varphi\colon\dbr\to(0,\infty)$ such that
\begin{equation}
	\label{delta4}
	F(K,s,t)=K\frac{\varphi(t)}{\varphi(s)},\quad K\geq0, (s,t)\in\Delta.
\end{equation}
\end{thm}

\begin{rem}
(\ref{delta4}) is the result of solving the multiplicative \textsc{Sincov} equation $f(s,t)f(t,u)=f(s,u)$. Moreover, choosing some fixed $t_0$, the function $\varphi$  is given by
\begin{equation}\label{phi}
	\varphi(t)=
	\begin{cases}
		f(t_0,t)&\text{, if } t\geq t_0\\
		\frac1{f(t,t_0)}& \text{, if } t<t_0.\\
	\end{cases}
\end{equation}
\end{rem}
Now we want to investigate the situation when (\ref{delta3}) is weakened in order to take care of (generalized) negative interest rates also in the situation when intervals of investments themselves  are considered rather than only the length of them.
\begin{thm}\label{thm4}
	The function $F\colon[0,\infty)\times\Delta\to[0,\infty)$
	satisfies the conditions \eqref{delta1} and \eqref{delta2} and
 \begin{align}
&F(\cdot,s,t)\text{ is monotonic on some interval for all }  (s,t)\in \Delta,\label{delta5}
	\end{align}
	iff there is some solution $f\colon\Delta\to[0,\infty)$ of the Sincov equation
	\begin{equation}\label{sincov}
	f(s,t)f(t,u)=f(s,u),\quad (s,t), (t,u)\in\Delta	
	\end{equation}
	such that
	\begin{equation}\label{gensol}
			F(K,s,t)=K f(s,t),\quad K\geq0, (s,t)\in\Delta.
	\end{equation}
\end{thm}
\begin{proof}
	(\ref{delta1}) and (\ref{delta5}) implies that $F$ has the form (\ref{gensol}) with $f(s,t)=F(1,s,t)$. Accordingly (\ref{delta2}) results in (\ref{sincov}).
	
	On the other hand (\ref{sincov}) and (\ref{gensol}) imply (\ref{delta1}), (\ref{delta2}) and (\ref{delta5}).
\end{proof}
\begin{rem}
	The rest of our considerations is devoted to the solution of (\ref{sincov}) in the slightly generalized situation that the function
	\begin{equation}
		f\colon\Delta_J\to\dbr
	\end{equation}
	is defined on $\Delta=\Delta_J:=\{(s,t)\in J^2\mid s\leq t\}$ for some non-trivial interval $J$, has $\dbr$ as the codomain and solves (\ref{sincov}). A special case has been considered in \cite[Theorem~14]{BFM2019}.
	The problem in its general form was posed by \textsc{Detlef Gronau} as Problem~2.1 in \cite{Grillhof2022}.
\end{rem}
From now on we assume that $f\colon\Delta\to\dbr$ satisfies the Sincov equation (\ref{sincov}) and we use $\Delta^\circ:=\{(x,y)\in\Delta\mid x<y\}$. And we proceed with some lemmata.
\begin{lem}\label{lem1}
	Assume that $(x,y)\in\Delta^\circ$ and that $f(x,y)\not=0$. Then
	\begin{equation}
	I_{(x,y)}:=\bigcup_{x'\leq x<y\leq y',\atop (x',y')\in\Delta^\circ, f(x',y')\not=0}[x',y']
	\end{equation}
is an interval, and $$f(u,v)\not=0, ~~~~\forall (u,v)\in\Delta \textrm{ satisfying }u,v\in I_{(x,y)}.$$
Moreover, $f(u,v)=1$,   if $u=v$.
\end{lem}
\begin{proof}
Of course $I_{(x,y)}$ is an interval since it is the union of a set of intervals with non empty intersection. Note also that $$0\not=f(x',y')=f(x',u)f(u,v)f(v,y')~~~~~~~(x'\leq u\leq v\leq y')$$ implies $f(u,v)\not=0$.

Now let $u,v\in I_{(x,y)}$, $u\leq v$ and $x_0\leq u\leq y_0$, $x_1\leq v\leq y_1$ where $(x_i,y_i)\in\Delta^\circ$, $x_i\leq x<y\leq y_i$ and $f(x_i,y_i)\not=0$ for $i=0,1$. Put $x_2:=\min(x_0,x_1), y_2:=\max(y_0,y_1)$. Then $x_2\leq u\leq v\leq y_2$. Thus it is enough to show $f(x_2,y_2)\not=0$. Let, for example $x_0\leq x_1$. Then $f(x_2,x)\not=0$ since $0\not=f(x_0,y_0)=f(x_0,x)f(x,y_0)$. Analogously we have $f(y,y_2)\not=0$ and therefore $f(x_2,y_2)=f(x_2,x)f(x,y)f(y,y_2)\not=0$.

$f(u,u)=f(v,v)=1$, since $0\not=f(u,v)=f(u,u)f(u,v)=f(u,v)f(v,v)$.
\end{proof}
\begin{lem}\label{lem2} Let $x,y$ be as in the previous lemma. Then
	\begin{align}
		I_{(x,y)}&=\bigcup_{I\in\mathcal{J}} I\text{,  where}\\
		\mathcal{J}&:=\mathcal{J}_{(x,y)}:=\{I\subseteq J\mid I \text{ is an interval},\\
		& x,y\in I, f(u,v)\not=0\text{ for all } u,v\in I\text{ such that } u<v \}.\notag
	\end{align}
\end{lem}
\begin{proof}
	Note that by Lemma~\ref{lem1} $I_{(x,y)}\in\mathcal{J}$ implying $I_{(x,y)}\subseteq\bigcup_{J\in\mathcal{J}} J$.
	
	Let, on the other hand $I\in\mathcal{J}$. Then there are sequences $(a_n), (b_n)$ such that $(a_n)$ is decreasing, $(b_n)$ is increasing, $a_n\leq x<y\leq b_n$ and $I=\bigcup_{n\in\dbn} [a_n,b_n]$. Since $a_n, b_n\in I$ the value $f(a_n,b_n)$  has to be $\not=0$. So $[a_n,b_n]\subseteq I_{(x,y)}$. Therefore $I\subseteq I_{(x,y)}$ for all $I\in\mathcal{J}$.
\end{proof}
\begin{lem}\label{lem3}
	Let $x,y$ be as above and assume that $u\in J\setminus I_{(x,y)}$. Then either
	\begin{align}
	\label{ulessv}&u<v  \text{ for all } v\in I_{(x,y)}\text{ or }\\
	\label{vlessu}&v<u  \text{ for all } v\in I_{(x,y)}.
	\end{align}	
Moreover in case (\ref{ulessv}) $f(u,v)=0$ and in case (\ref{vlessu}) $f(v,u)=0$.
\end{lem}
\begin{proof}
	Assume $v\leq u\leq w$ for some $v,w\in I_{(x,y)}$. The $u\in[v,w]\subseteq I_{(x,y)}$, a contradiction.
	
	So, let $u<v$ for all $v\in I_{(x,y)}$, and suppose that $f(u,v_0)\not=0$ for some $v_0\in I_{(x,y)}$. Then there are $x_0,y_0\in I_{(x,y)}$ such that $x_0\leq v_0\leq y_0$, $x_0\leq x<y\leq y_0$ and $f(x_0,y_0)\not=0$.
Then $f(u,y_0)=f(u,v_0)f(v_0,y_0)\not=0$ since $f(u,v_0), f(v_0,y_0)\not=0$.
But then $[u,y_0]\subseteq I_{(x,y)}$ contradicting $u\not\in I_{(x,y)}$.

The other case, $v<u$ for all $v\in I_{(x,y)}$ may be treated similarly.
\end{proof}
\begin{lem}\label{lem4}
	Let $(x,y)\in \Delta^\circ$ be such that $f(x,y)\not=0$. Then $I_{(u,v)}=I_{(x,y)}$ for all $u,v\in I_{(x,y)}$ satisfying $u<v$.
\end{lem}
\begin{proof}
	By Lemma~\ref{lem1} we have $f(u',v')\not=0$ for all $u', v'\in I_{(x,y)}$ with $u'<v'$. Thus using Lemma~\ref{lem2} results in $I_{(x,y)}\subseteq I_{(u,v)}$.  Therefore $x,y\in I_{(u,v)}$, which analogously implies $I_{(u,v)}\subseteq I_{(x,y)}$.
\end{proof}
\begin{lem}\label{lem5}
Let $I_1, I_2$ be intervals such that $I_1\cap I_2=\{a\}$. Then either $a=\min(I_1)=\max(I_2)$ or $a=\min(I_2)=\max(I_1)$.
\end{lem}
\begin{proof}

The (simple) considerations are left to the reader.
\end{proof}
\begin{lem}\label{lem6}
	Let $(x,y), (u,v)\in\Delta^\circ$ with $f(x,y), f(u,v)\not=0$. Then either $I_{(x,y)}=I_{(u,v)}$ or $I_{(x,y)}\cap I_{(u,v)}=\emptyset$.
\end{lem}
\begin{proof}
	Let $I_{(x,y)}\cap I_{(u,v)}\not=\emptyset$. If this intersection contains two different points $a,b$ with, say, $a<b$, then $I_{(x,y)}=I_{(a,b)}=I_{(u,v)}$ by Lemma~\ref{lem4}. Otherwise $I_{(x,y)}\cap I_{(u,v)}=\{a\}$. Using Lemma~\ref{lem5} we may assume, without loss of generality that  $a=\max I_{(x,y)}=\min I_{(u,v)}$.
	Therefore $x<y\leq a\leq u<v$ and $f(x,a), f(a,v)\not=0$. Thus also $f(x,v)=f(x,a)f(a,v)\not=0$ which implies by Lemma~\ref{lem2} that $[x,v]\subseteq I_{(x,y)}\cap I_{(u,v)}$. So $I_{(x,y)}\cap I_{(u,v)}=\{a\}$ is not possible. Accordingly $I_{(x,y)}\cap I_{(u,v)}\not=\emptyset$ implies $I_{(x,y)}= I_{(u,v)}$.
\end{proof}
Now we are able to formulate necessary conditions for the solutions $f$ of the Sincov equation on $ \Delta=  \Delta_J$.
\begin{thm}\label{thm5}
	Let $f\colon \Delta=\Delta_J\to\dbr$ be a solution of (\ref{sincov}). Then there is a countable (possibly empty) system $\mathcal{S}$ of pairwise disjoint non-trivial intervals $I\subseteq J$ and there is a function $d\colon \bigcup_{I\in \mathcal{S}}I  \to \dbr^\times:=\dbr\setminus\{0\}$, such that
	\begin{equation}
	f(x,y)=\frac{d(y)}{d(x)},\quad x, y\in I\in \mathcal{S}, x\leq y.
	\end{equation}
Moreover for any $x\in I\in \mathcal{S}$
\begin{equation}\label{equalzero}
	f(x,y)=f(z,x)=0,\quad I\not\ni z<x<y\not\in I.
\end{equation}
and \begin{equation}\label{eqfo1}
  f(x,x)=\begin{cases}
    1& \textrm{if }x\in \bigcup_{I\in \mathcal{S}}I,\\
    0 \textrm{ or }1 & \textrm{otherwise.}
\end{cases}\end{equation}
\end{thm}
\begin{proof} Let $\mathcal{S}:=\{I_{(x,y)}\mid (x,y)\in\Delta^\circ, f(x,y)\not=0\}$. Then the intervals in $\mathcal{S}$ are pairwise disjoint by Lemma~\ref{lem6}. Moreover $\mathcal{S}$ is countable since every $I_{(x,y)}$ equals  $I_{(r,s)}$ with $x\leq r<s<y$ and $r,s\in\dbq$ by Lemma~\ref{lem4}.

Let $I\in\mathcal{S}$ and $x_0\in I$. Then $d\colon I\to\dbr^\times$,
\begin{equation}\label{eq29}
	d(x)=
	\begin{cases}
		f(x_0,x)&\text{, if } x\geq x_0\\
		\frac1{f(x,x_0)}& \text{, if } x<x_0,\\
	\end{cases}	
\end{equation}
is well defined and satisfies $f(x,y)=\frac{d(y) }{d(x)}$ for all $x,y\in I, x\leq y$. This can be easily seen by distinguishing the cases $x\geq x_0$, $y<x_0$ and $x< x_0\leq y$.
	
(\ref{equalzero}) follows from Lemma~\ref{lem3}.

Since $f(x,x)f(x,y)=f(x,y)\ne 0$ for $x,y\in I\in\mathcal{S}$, hence $f(x,x)=1$ which is the first case of \eqref{equalzero}. The second case of \eqref{equalzero} follows from $f(x,x)=f(x,x)f(x,x)$ and from Lemma~\ref{lem3}. (Note also that $f(x,y)=0$ for all $x,y\not\in \bigcup_{I\in \mathcal{S}} I$ with $x<y$.)
\end{proof}
Finally we prove that all Sincov functions on $ \Delta=  \Delta_J$ may be obtained by using the just derived necessary conditions.
\begin{thm}\label{thm6}
	Let $J$ be a non-trivial interval and $\mathcal{S}$ at most countable (maybe empty) set of disjoint subintervals of $J$. Let furthermore $d\colon \bigcup_{I\in \mathcal{S}}I  \to \dbr^\times$ be an arbitrary function. Then $f\colon \Delta=\Delta_J\to\dbr$ defined by
	\begin{equation}
		f(x,y):=
		\begin{cases}
			\frac{d(y)}{d(x)}, & \textrm{if } x,y\in I\in\mathcal{S},\\
			0, &\text{if } x,y \textrm{ not in the same } I\in \mathcal{S},\\
            0 \textrm{ or }1 \textrm{ arbitrarily}, & \textrm{ if } x=y \not\in  \bigcup_{I\in \mathcal{S}}I
		\end{cases}
	\end{equation}
satisfies (\ref{sincov}).
\end{thm}
\begin{proof}
	Let $x\leq y\leq z$. If $x,z\in I(\in \mathcal{S})$ then also $y\in I$.
	Moreover $$f(x,y)f(y,z)=\frac{d(y)}{d(x)}\frac{d(z)}{d(y)}=\frac{d(z)}{d(x)}=f(x,z).$$
	If $x\in I$ and $z\not\in I$ we have $f(x,z)=0$. Assuming  $y\in I$ implies $f(y,z)=0$ and therefore $0=f(x,z) =f(x,y) f(y,z)$. This also holds true when also $y\not\in I$.
	
	If $x\not\in \bigcup_{I\in \mathcal{S}} I$ we have $f(x,z)=0$ for $x<z$ and therefore $0=f(x,z)=0\cdot f(y,z)=f(x,y)f(y,z)$ if additionally $x<y$.
	In case $y=x$ we have $f(y,z)=0$ implying (\ref{sincov}). If finally $x=y=z$ we again have $f(x,z)=f(x,y) f(y,z)$.
\end{proof}
\begin{rem}
	Gronau in \cite{Grillhof2022} gave two types of solutions. The first one with $f(x,y)=\delta_x(y)$ for all $x\leq y$ is the special case $\mathcal{S}=\emptyset, g=1$ of Theorem~\ref{thm6}. The second one may be described by $\mathcal{S}=\{[x_0,y_0]\}$, $\sigma=d_{[x_0,y_0]}$ and $g(x)=0$ for all $x\not\in[x_0,y_0]$.
\end{rem}

\begin{rem}[Generalization] The codomain of the function $f$ may be chosen much more general without altering the results.
	
Let $(G,\cdot)$ be an arbitrary non necessarily abelian group with neutral element $1$ and add an \emph{absorbing} element $0$, such that $0\not\in G$  and in $G':=G\cupdot\{0\}$ we have $x\cdot 0=0\cdot x=0$. Then the only elements in $G'$ with $x^2=x$ are $0$ and $1$. This, more or less, implies that Theorems~\ref{thm5},~\ref{thm6} also hold in the new situation with the modification that the functions $d$ are defined as
\begin{equation}\label{eq29'}
	d(x)=
	\begin{cases}
		f(x_0,x)&\text{, if } x\geq x_0\\
		f(x,x_0)^{-1}& \text{, if } x<x_0,\\
	\end{cases}	
\end{equation}
because then $f(x,y)=d(x)^{-1}d(y)$ in Theorem~\ref{thm5}
and
\begin{equation*}
	f(x,y)f(y,z)=	d(x)^{-1}d(y)d(y)^{-1}d(z)=d(x)^{-1}d(z)=f(x,z)
\end{equation*}
in the proof of Theorem~\ref{thm6}.

Examples for groups with added absorbing elements are $G\cupdot\{0\}$ where $G$ is subgroup of $K^\times$ for division algebras $K$, in particular $\dbr$ and $[0,\infty)$ and also,for any $n$ and any field $K$, the union  $G\cupdot\{0\}$ where $G$ is a subgroup $\text{Gl}_n(K)$ and $0$ the null matrix. The last example itself is a special case of $G\cupdot\{0\}$ where $G$ is a subgroup of the group of units in a unitary ring $R$ and $0$ the zero element in $R$.
\end{rem}
\begin{rem}
	In \cite{BFM2019} the Sincov equation is considered in the form
	\begin{equation}\label{FechnerSincov}
		g(x,z)=g(x,y)g(y,z),\quad x>y>z, x,y,z\in J,
	\end{equation}
with $J=(0,1)$.
The general solution of (\ref{FechnerSincov}) is easily derived from Theorems~\ref{thm5} and~\ref{thm6} by
\begin{enumerate}
	\item[i)] considering $f$ defined by $f(y,x)=g(x,y)$ when  $x>y$  and
	\item[ii)] by observing that we may extend $f$ to the pairs $(x,x)$ by choosing $f(x,x)\in\{0,1\}$ as in Theorems \ref{thm5} and \ref{thm6}.
\end{enumerate}

\end{rem}
As a result of Theorems \ref{thm6} and \ref{thm5} we obtain the following characterization of $F(K,s,t)$ satisfying \eqref{delta1}, \eqref{delta2} and \eqref{delta5}.
\begin{cor}
    The function $F\colon[0,\infty)\times\Delta_J\to[0,\infty)$
satisfies the conditions \eqref{delta1}, \eqref{delta2} and \eqref{delta5}
iff there is a countable (possibly empty) system $\mathcal{S}$ of pairwise disjoint non-trivial intervals $I\subseteq J$ and there is a function $d\colon \bigcup_{I\in \mathcal{S}}I  \to \dbr^\times$, such that
	\begin{equation}
	F(K,s,t)=K\frac{d(s)}{d(t)},\quad s, t\in I\in \mathcal{S}, s\leq t.
	\end{equation}
Moreover for any $s\in I\in \mathcal{S}$
\begin{equation}\label{equalzero1}
	F(K,s,t)=F(K,u,s)=0,\quad I\not\ni u<s<t\not\in I.
\end{equation}
and \begin{equation}\label{eqfox}
  F(K,s,s)=\begin{cases}
    K& \textrm{if }s\in \bigcup_{I\in \mathcal{S}}I,\\
    0 \textrm{ or }K & \textrm{otherwise.}
\end{cases}\end{equation}
\end{cor}
\begin{rem}
This corollary also implies Theorem~\ref{thm3} by observing that \eqref{delta3}
implies $\mathcal{S}=\{\mathbb{R}\}$ and also that  $d$ is monotonically increasing.
\end{rem}

\section*{Acknowledgement} We would like to express our gratitude to the anonymous referee for the valuable comments, which help to clarify certain points and increase the quality of the paper.


\begin{tabular*}{\textwidth}{l@{\extracolsep{\fill}}cl}
	Gergely Kiss, & & Jens Schwaiger \\
	Alfréd Rényi Institute of Mathematics & & Institut f\"{u}r Mathematik\\
	 & & Karl-Franzens Universit\"at
	Graz\\
	Re\'altanoda utca 13-15 & & Heinrichstra{\ss}e 36\\
	H-1053 Budapest & & A-8010 Graz\\
	Hungary & & Austria\\
	\texttt{kigergo57@gmail.com} & &
	\texttt{jens.schwaiger@uni-graz.at}
\end{tabular*}
\end{document}